\documentclass{amsart}

\usepackage{amsmath,amsfonts,amssymb,amsthm,amscd,comment,euscript}
\usepackage{mathbbol, stmaryrd}

\usepackage[all]{xy}
\usepackage{graphicx}
\usepackage{enumerate} 
\usepackage[colorlinks=true]{hyperref} 
\usepackage[usenames,dvipsnames]{xcolor}

\usepackage{hyperref} 

\allowdisplaybreaks

\theoremstyle{plain}
\newtheorem{lem}{Lemma}[section]
\newtheorem{cor}[lem]{Corollary}
\newtheorem{prop}[lem]{Proposition}
\newtheorem{thm}[lem]{Theorem}

\theoremstyle{definition}
\newtheorem{ex}[lem]{Example}
\newtheorem{rem}[lem]{Remark}
\newtheorem{dfn}[lem]{Definition}

\newcommand{\bbZ}{\mathbb{Z}}

\newcommand{\calP}{\mathcal{P}}
\newcommand{\calL}{\mathcal{L}}

\newcommand{\al}{\alpha}

\newcommand{\de}{\delta}
\newcommand{\la}{\lambda}
\newcommand{\La}{\Lambda}

\newcommand{\im}{\operatorname{im}}            

\DeclareMathOperator{\pt}{pt}   

\newcommand{\bfD}{\mathbf{D}}

\usepackage{tikz-cd}

\usepackage{ amssymb }

\newcommand{\hX}{\hat X}

\newcommand{\llbr}{\llbracket}
\newcommand{\rrbr}{\rrbracket}
\newcommand{\bbP}{\mathbb{P}}
\newcommand{\bfc}{\mathbf{c}}

\DeclareMathOperator{\rev}{rev}
\newcommand{\bj}{\mathbf{j}}
\usepackage{ stmaryrd }
\title{On equivariant oriented cohomology of Bott-Samelson varieties}
\author{Hao Li}
\address{State University of New York Albany, ES 110, 1400 Washington Ave, Albany, NY 12222}
\email{hli29@albany.edu}
\author{Changlong Zhong}
\email{czhong@albany.edu}

\begin{document}
\maketitle
\begin{abstract}
For any Bott-Samelson resolution  $q_{I}:\hX_{I}\rightarrow G/B$ of the flag variety $G/B$, and  any  torus equivariant oriented cohomology $h_T$, we compute  the restriction formula of certain basis $\eta_L$ of $h_T(\hX_{I})$ determined by the projective bundle formula. As an application, we show that $h_T(\hX_{I})$ embeds into the equivariant oriented cohomology of $T$-fixed points, and the image can be characterized by using the Goresky-Kottwitz-MacPherson (GKM) description.  Furthermore, we compute   the push-forward  of the basis $\eta_L$ onto $h_T(G/B)$, and their restriction formula. 
\end{abstract}
\section{Introduction}
Let $G/B$ be a flag variety. For each $w$ in the Weyl group $W$, and a reduced decomposition $w=s_{i_1}\cdots s_{i_l}$, one defines the variety 
\[
\hX_{I_w}=P_{i_1}\times ^BP_{i_2}\times^B\cdots \times^B P_{i_l}/B.
\]
Here $P_{i_j}$ is the minimal parabolic subgroup corresponding to the simple root $\al_{i_j}$. Multiplication of all the coordinates defines a canonical map $q_{I_w}:\hX_{I_w}\to G/B$, which is proper and birational over the Schubert variety $X(w)$ of $w$.   This is called a Bott-Samelson resolution of $X(w)$. These resolutions play  important role in Schubert calculus and representation theory. 

We are interested in $h_T(\hX_{I_w})$, where $h_T$ is an (equivariant) oriented cohomology theory in the sense of Levine-Morel. Examples of $h$ include  the Chow group (singular cohomology) and K-theory. For any $h$, it is proved in \cite{CZZ1, CZZ2, CZZ3} that, after fixing a reduced decomposition $I_w$ for each $w\in W$, the push-forward $q_{I_w*}(1)$ in $h_T(G/B)$ of the fundamental class define a basis of $h_T(G/B)$ over the base ring $h_T(\pt)$. This enables the authors of loc.it. in constructing the algebraic replacement of $h_T(G/B)$, and provides standard setting for generalized Schubert calculus. For further study on equivariant oriented cohomology of $T$-varieties following this method, please refer to \cite{DZ20, GZ20, LZZ16, CNZ19, Z20}.

Let us  consider $h_T(\hX_{I})$ for a general sequence $I=(i_1,...,i_l)$. Note that the set $\hX_I^T$ of $T$-fixed points of $\hX_{I}$ is in bijection to the power set of $[l]=\{1,2,...,l\}$. Denote by $\bj:\hX_{I}^T\to \hX_{I}$ the canonical embedding. Our main result is the following:
\begin{thm} \label{thm:main}(Corollary \ref{cor:inj}) For any sequence $I$, the pull-back to $T$-fixed points $\bj^{*}:h_T(\hX_I)\to h_T(\hX_I^T)$ is injective.
\end{thm}
Furthermore, we show that elements in the the image of $\bj^*$ satisfy the Goresky-Kottwitz-MacPherson (GKM) description (see Theorem \ref{thm:image}). Indeed, in case the sequence $I=(i_1,...,i_l)$ consists of distinct $i_j$'s, we prove that  the GKM description uniquely characterizes the image (Theorem \ref{thm:GKM}).

Let us mention the idea of the proof briefly. Since $\hX_{I}$ is constructed as a tower of $\bbP^1$-bundles, there is a canonically defined algebra generators $\eta_j\in h_T(\hX_{I})$ corresponding to each parabolic subgroup $P_{i_j}$ in $\hX_{I}$. Each $\eta_j$ satisfies certain quadratic relation. Therefore, for each subset  $L$ of  $[l]$,  denoting by $\eta_L$  the product of $\eta_j$ with $j$ in  $L$, then $\eta_L$ form a basis of $h_T(\hX_{I})$.  

We compute the restriction  $\bj^*(\eta_L)$ explicitly (Theorem \ref{thm:res}). The computation uses the characteristic map $\bfc:h_T(\pt)\to h_T(\hX_I)$ induced by the map sending a character $\la$ of $T$ to the first Chern class of the associated line bundle over $\hX_I$.  We then use the explicit formula of $\bj^*(\eta_L)$ to prove Theorem \ref{thm:main}, and the relationship of the image with the GKM description. 

As another application of the computation of $\bj^*(\eta_L)$, we also compute the push-forward of $\eta_L$ via the canonical map $q_I:\hX_I\to G/B$. We show that the push-forward $q_{I*}(\eta_L)$ coincides with the Bott-Samelson class corresponding to the sequence $I\backslash L$. 

For future applications, one would apply the restriction formula (Theorem \ref{thm:res}) and the push-forward formula (Theorem \ref{thm:pushf}) in the study of motivic Chern (mC) classes in K-theory. MC classes are certain K-theory classes associated to constructible subsets of $T$-varieties. For details, please refer \cite{AMSS17,RTV15, RTV17}. They are closely related with the K-theoretic stable basis of Springer resolutions, defined by Maulik-Okounkov \cite{MO12, O15} and studied in \cite{SZZ17, SZZ19}. Indeed, Mihalcea  has some recent work on the relationship between push-forward of MC classes of Bott-Samelson varieties and the Kazhdan-Lusztig basis of Hecke algebra. The authors hope to apply the computation of this paper on understanding this relationship.

The paper is organized as follows: In Section \ref{sec:2} we recall necessary notions of equivariant oriented cohomology theory, formal group algebra, and the characteristic map $\bfc$. In Section \ref{sec:3} we recall some basic facts about Bott-Samelson varieties. In Section \ref{sec:4} we compute the restriction formula (Theorem \ref{thm:res}), which was used to prove the injectivity of the pull-back map $\bj^*$ and the GKM description (Theorem \ref{thm:image}). In Section \ref{sec:5} we compute the push-forward of the basis $\eta_L$ onto $h_T(G/B)$. 





{\it Acknowledgments:}
The second author would like to thank Leonardo Mihalcea and Rebecca Goldin for helpful conversations. 

\section{Equivariant oriented cohomology theory} \label{sec:2}
In this section, we define some notations, and collect some basic notions and facts about equivariant oriented cohomology theory. 

Let $G$ be a split semisimple linear algebraic group over a field $k$, with rank $n$. Let $T$ be a split maximal torus of $G$ and  $B\subset  G$ be a Borel subgroup. Let $\Pi=\left\{\alpha_{1}, \alpha_{2},\ldots, \alpha_{n}\right\} $ be the set of simple roots, and $\Sigma$ be the set of roots. Let $P_{i}$ be the minimal parabolic subgroup corresponding to the simple root $\alpha_{i}$. The Weyl group $W$ of  $G$ is generated by $\left\{s_{\alpha_{1}},\ldots, s_{\alpha_{n}}\right\}$ where $s_{\alpha_{i}}$ is the reflection corresponding to $\alpha_{i}$. Note that $W$ can be identified with $N_G(T)/T$. Sometimes we will understand $s_i=s_{\alpha_{i}}$ as an element in $G$. We denote the group of characters of $T$ by $\Lambda$. For each positive integer $l$, denote $[l]=\{1,2,...,l\}$. 
 
 Let $F$ be a formal group law over the commutative ring $R$. Examples include the additive formal group law $F_a=x+y$ over $\bbZ$, and the multiplicative formal group law $F=x+y-\beta xy$ over $\bbZ[\beta, \beta^{-1}]$.

\begin{dfn}
Let $R\llbr x_\La\rrbr:=R\llbr x_{\lambda}|\la\in \La\rrbr$ be the power series ring.  Let $J_{F}$ be the closure of the ideal generated by $x_{0}$ and $x_{\lambda+\mu}-F(x_{\lambda}, x_{\mu}), ~\la, \mu\in \La $. We define the formal group algebra $R\llbracket \Lambda \rrbracket_{F}$ to be the quotient 
$$R\llbracket \Lambda \rrbracket_{F}=R\llbracket x_{\Lambda} \rrbracket/J_{F}.$$
\end{dfn}
Indeed, $R\llbr \La \rrbr_F$ is non-canonically isomorphic to the formal power series ring with $n$ variables. For simplicity, we denote $S=R\llbr \La \rrbr_F$. Note that by definition, $x_{-\la}$ is the formal inverse of $x_\la$, that is, $F(x_\la, x_{-\la})=0$. Since any formal group law $F$ is always of the form
\[
F(x,y)=x+y+a_{11}xy+\text{ higher order terms}, \quad a_{11}\in R,
\]
so it is not difficult to see that $x_{-\la}=-x_\la+x_\la^2 f(x_\la)$ for some $f(t)\in R[[t]]$. Therefore, $\frac{x_\la}{x_{-\la}}$ is an invertible element in $S$. 
\begin{ex}
\begin{enumerate}
\item 
Let $F_{a}$ be additive formal group law, then we have a ring isomorphism $$R\llbracket \Lambda \rrbracket_{F_{a}}\cong S_{R}(\Lambda)^{\wedge}, \quad x_{\lambda}\mapsto \lambda, $$ where $S_R(\Lambda)$ is the symmetric algebra of $\Lambda$ and the completion is done at the augmentation ideal.
\item 
Let $R[\Lambda]$  be the group algebra $\left\{\sum_{j}a_{j}e^{\lambda_{j}}|a_{j}\in R, \lambda_{j}\in \Lambda\right\}$. Then we have isomorphism $$R\llbracket \Lambda \rrbracket_{F_{m}}\cong R[\Lambda]^{\wedge}, \quad x_\la\mapsto  \beta^{-1}(1-e^{\lambda}),$$
where the completion $^\wedge$ is done at the augmentation ideal. 
\end{enumerate}
\end{ex}

Throughout this paper, suppose the root datum of $G$ together with the formal group law $F$ satisfy the regularity condition of \cite[Definition 4.4]{CZZ1}. For example, this is satisfied if $2$ is regular in $R$. Please consult loc.it. for more detail.  In particular, $x_\al$ is regular in $S$, for any root $\al$. 
The Weyl group action on $\Lambda$ induces an action of $W$ on $R\llbracket \Lambda \rrbracket_{F}$  by $s_{\alpha}(x_{\lambda})=x_{s_{\alpha}(\lambda)}.$ In particular, we have

\begin{lem}\cite[Corollary 3.4]{CPZ13} \label{lem:div}
 For any $v,w\in W$, any root $\al$ and $p\in S$,  we have 
\[
\frac{vs_{\al}w(p)-vw(p)}{x_{v(\al)}}\in S.
\]
\end{lem}
\begin{proof}According to \cite[Corollary 3.4]{CPZ13}, we know that $s_{\al}w(p)-w(p)$ is uniquely divisable by $x_{\al}.$ In other word, \[\frac{s_{\al}w(p)-w(p)}{x_{\al}}\in S.\] Then \[v(\frac{s_{\al}w(p)-w(p)}{x_{\al}})=\frac{vs_{\al}w(p)-vw(p)}{x_{v(\al)}}\in S.\]
\end{proof}
In particular, taking $w=v=e$, we see that $x_{\al}|(p-s_\al(p))$. We can then define the Demazure operator $\Delta_\al: S\to S$ by 
\begin{equation}\label{eq:Delta}
\Delta_\al(p)=\frac{p-s_\al(p)}{x_\al}.
\end{equation}

\begin{rem}
By direct calculation, we have the following formulas:  for $p,q \in S$, 
 \begin{align} \label{eq:m} s_{\alpha}\Delta_{\alpha}(p)=-\Delta_{-\al}(p) \end{align}
\begin{align} \label{eq:n} \Delta_{\al}(pq)=\Delta_{\alpha}(p)q+p\Delta_{\alpha}(q)-\Delta_{\alpha}(p)\Delta_{\alpha}(q)x_{\al}. \end{align}
\end{rem}

Let $h_{T}$ be an equivariant oriented cohomology theory of Levine-Morel. Roughly speaking, it is an additive contravariant functor $h_{T}$ from the category of smooth projective  $T$-varieties to the category of  commutative rings with units,  satisfying the following axioms: existence of push-forwards for projective morphisms, homotopy invariance and the projective bundle axioms \cite[\S2]{CZZ3}.  The Chern classes are defined. Moreover, there exists a formal group law $F$  over $R=h_{T}(\pt)$ such that if $\mathcal{L}_{1}$ and $\mathcal{L}_{2}$ are locally free sheaves of rank one, then $$c_{1}(\mathcal{L}_{1}\otimes\mathcal{L}_{2})=F(c_{1}(\mathcal{L}_{1}), c_{1}(\mathcal{L}_{2})). $$

It is proved in \cite[Theorem 3.3]{CZZ3} that 
\[
S=R\llbr\La\rrbr_F\cong h_T(\pt), \quad x_\la\mapsto c_1(\calL_\la),
\]
where $\calL_\la$ is the associated line bundle. As an immediate consequence, we see that if the variety $X$ is a finite set of points (with trivial $T$-action), then 
\[
h_T(X)=F(X;S),
\]
where the latter is the set of all maps from $X$ to $S$. It has a $S$-basis $f_x, x\in X$, and is a ring with product defined by 
\[
f_x\cdot f_y=\de_{x,y}f_x,\quad \text{ and  unit} \sum_{x\in X}f_x. 
\]

By functoriality, if $p:X\to Y$ is a $T$-equivariant map between two finite discrete sets of points on which $T$ acts trivially, then
\begin{align}\label{eq:finite}
p^*(f_y)=\sum_{x\in f^{-1}(y)}f_x, \quad p_*(f_x)=f_{p(x)}. 
\end{align}

We recall the definition of the characteristic map. 
Let $X$ be a $T$-variety on which $B$ acts on the right, and the $T$ and $B$ actions are commutative. Moreover, suppose  the quotient $X/B$ exists and $X\to X/B$ is a $T$-equivariant principal bundle. Following \cite[\S10.2]{CPZ13}, we can define a ring homomorphism
\[\bfc:S=h_T(\pt)\to h_T(X/B), \quad x_\la\mapsto c_1(\calL_\la). \]
It is called the characteristic map. 

Let $\al$ be a simple root with corresponding minimal parabolic subgroup $P_\al$.  Consider the fiber product $X'=X\times^B P_{\al}$, then $X'$ is a $T$-equivariant principal $P_\al$-bundle over $X/B$. Denote $p:X'/B\to X/B$, and there is a zero section 
\[\sigma: X/B\to X'/B, x\mapsto (x,1).\]  Similar as \cite[\S10.5]{CPZ13}, we have 
\begin{align}
h_T(X'/B)\cong h_T(X/B)[\xi]/(\xi^2-y\xi), \quad \xi=\sigma_*(1), \quad y=p^*\sigma^*\xi.
\end{align}

The following properties can be proved similarly as the non-equivariant version in \cite[\S10]{CPZ13}. 
\begin{lem}\label{chara} Denote $\bfc:S\to h_T(X/B)$ and $\bfc':S\to h_T(X'/B)$. For each $\la\in \La$, denote the associated line bundles on $X/B$ and $X'/B$  by $\calL_\la$ and $\calL'_\la$, respectively. 
\label{lem:char}\begin{enumerate}
\item We have $\sigma^*\xi=c_1(\calL_{-\al})=\bfc(x_{-\al})$. 
\item $y=p^*\sigma^*\xi=p^*\bfc(x_{-\al})$. 
\item For any $u\in S$, we have 
\[\sigma^*\bfc'(u)=\bfc(u), \quad \bfc'(u)=p^*\bfc(s_\al(u))+p^*\bfc(\Delta_{-\al}(u))\cdot \xi.\]
\end{enumerate}
\end{lem}

\section{Bott-Samelson Varieties}\label{sec:3}

In this section, we collect some facts about Bott-Samelson varieties .

\begin{dfn}
For any sequence  $I=(i_{1},i_{2},\ldots, i_{l})$ with $1\leq i_{j} \leq n$, we define the variety $\hX_{I}$ to be 

\[\hX_{I}=P_{i_{1}}\times^{B}P_{{i_{2}}}\times^{B}\ldots \times^{B} P_{i_{l}}/B.\]
Here the right $B$-action is given by right multiplication on the last coordinate. 
 If $I=\emptyset$, then we set $\hX_{\emptyset}=\pt.$ The variety  $\hX_{I}$ is called the Bott-Samelson variety corresponding to $I$. It has an obvious $T$-action by left multiplication on the first coordinate. 
\end{dfn}

Since $P_i/B\cong \bbP^1$, so we have a  sequence of $\mathbb{P}^{1}-$bundles:
\begin{equation}\label{eq:P1bundle}\xymatrix{\hX_{I} \ar[r] & \hX_{(i_{1},\ldots,i_{l-1})} \ar[r] \ar[r] \ar@/^6pt/[l]^-{\sigma_{l}} &\ldots  \ar@/^6pt/[l]^-{\sigma_{l-1}}\ar[r] & \hX_{(i_{1})} \ar[r] \ar@/^6pt/[l]^-{\sigma_{2}} & pt\ar@/^6pt/[l]^-{\sigma_{1}} , }
\end{equation} where $\sigma_{i}$, $1\leq i \leq l$ are the zero sections. Multiplication of all factors of $\hX_{I}$ induces a  map $$q_{I}:\hX_{I}\rightarrow G/B.$$
Denote by $\pi_i:G/B\to G/P_i$ the canonical map, and denote $I'=(i_1,...,i_{l-1})$.  We then have the following Cartesian diagram:
\begin{align}\label{eq:basechange} \xymatrix{\hX_{I} \ar[rr]^{q_{I}} \ar[d]^{p}& & G/B \ar[d]^{\pi_{\alpha_{l}}}  \\ \hX_{I'} \ar[rr]^{\pi_{\alpha_{l}}\circ q_{I'}}& & G/P_{\alpha_{l}}}.
\end{align}
So we have the  base-change formula ${q_{I}}_{*} p^{*}=\pi_{\alpha_{l}}^{*}(\pi_{\alpha_{l}} q_{I'})_{*}$. The operator 
\[\pi_{\alpha_{l}}^{*}\pi_{\alpha_{l}*}:h_T(G/B)\to h_T(G/B)\]
 is called the push-pull operator.

Denote by $\bfc_I: S\to \hX_I$ the characteristic map. The following proposition describes the $R$-algebra structure of equivariant oriented cohomology of Bott-Samelson varieties.
\begin{prop}\cite[\S11.3]{CPZ13} \label{prop:ring}
We have the following presentation
\[h_{T}(\hX_{I})\cong h_{T}(pt)[\eta_{1},\eta_{2},\ldots,\eta_{l}]/{(\left\{{\eta_{j}}^{2}-{y}_{j}{\eta}_{j}|{j=1,\ldots,l}\right\}}),
\] where 
\[y_{j}=p^{*}\bfc_{(i_{1},\ldots,i_{j-1})}(x_{-\alpha_{i_{j}}}), \quad \eta_{j}=p^{*}{\sigma_{j}}_{*}(1), \]
with $p^{*}$  the pull-back from $h_T(\hX_{(i_1,..., i_j)})$ to $h_{T}(\hX_{I})$. 
\end{prop}

For ordinary oriented  cohomology, this theorem is proved in \cite{CPZ13}. The idea of the proof is to apply the projective bundle formula to the sequence of  $\mathbb{P}^{1}-$bundle \eqref{eq:P1bundle}. One can check that all the arguments hold in the equivariant setting, which can be used to prove Proposition \ref{prop:ring}. 

For each subset $L\in \calP_l$, define 
\[
\eta_L=\prod_{j\in L}\eta_{j}\in h_T(\hX_I). 
\]
\begin{cor} The $S$-module $h_{T}(\hX_I)$ is free with basis $\left\{\eta_L|L\in \calP_l \right\}.$
\end{cor}
Since in Proposition \ref{prop:ring}, the $y_j$ does not belong to the coefficient ring $h_T(\pt)$, so the presentation of $h_T(\hX)$ is not satisfactory. To get a polynomial presentation of it, we follow the idea in \cite[Theorem 11.4]{CPZ13}. 
\begin{lem}\label{lem:ind}
For any sequence $I=(i_1,...,i_l)$, we have 
\[
\bfc_{I}(u)=\sum_{L\subset [l]} \theta_{I,L}(u)\eta_{L}, \quad u\in S, 
\]
where $\theta_{l,L}=\theta_{1}\cdots \theta_{l}$ with $\theta_{j}= \begin{cases} \Delta_{-\alpha_{i_{j}}}, & \mbox{if } j\in L  , \\ s_{{i_{j}}}, &  \mbox{ otherwise}.  \end{cases}$
\end{lem}
\begin{proof}We use induction on $l$. If $l=1$, from Lemma \ref{lem:char}, we have  $$\bfc_{i_{1}}(u)=p^*\bfc_{\emptyset}(s_{{i_{1}}}(u))+p^{*}\bfc_{\emptyset}(\Delta_{-\al_{i_{1}}}(u))\cdot \eta_{1}.$$ 
Note that the characteristic map $\bfc_{\emptyset}:S\rightarrow h_{T}(pt)$ is the identity map. So it holds.

Now assume the conclusion holds for $I':=(i_{1},\ldots, i_{l-1})$. Denote $p:\hX_I\to \hX_{I'}$. By Lemma \ref{lem:char} we have
\begin{align*} \bfc_{I}(u)&=p^{*}\bfc_{I'}(s_{{i_{l}}}(u))+p^{*}\bfc_{I'}(\Delta_{-\al_{i_{l}}}(u))\cdot \eta_{l} \\
&=\sum_{L\subset [l-1]}\theta_{l-1,L}(s_{i_{l}}(u)) \eta_{L}+\sum_{L\in [l-1]}\theta_{l-1, L}(\Delta_{-\al_{i_j}}(u)) \eta_{L}\cdot \eta_{l}\\
&=\sum_{L\subset [l] }\theta_{l,L}(u)\eta_L.
\end{align*}
\end{proof}
\begin{prop}\cite[Theorem 11.4]{CPZ13}
The ring $h_{T}(\hX_I)$ is a quotient of the polynomial ring $S[\eta_{1},\eta_{2},\ldots,\eta_{l}] $ modulo the relations 
\[\eta_{j}^{2}=\sum_{L\subset [j-1]}\theta_{j-1,L}(x_{-\alpha_{i_{j}}})\eta_{L}\eta_{j}, \quad j\in [l].\]
\end{prop}

\begin{proof}
Denote $K=(i_1,...,i_{j-1})$ and $p:\hX_I\to \hX_K$. By definition of $y_j$ and Lemma \ref{lem:ind}, we have 
\[
y_{j}=p^{*}\bfc_{K}(x_{\al_{i_{j}}})=p^*(\sum_{L\subset [j-1]}\theta_{j-1,L}(x_{-\alpha_{i_{j}}})\eta_{L})=\sum_{L\subset [j-1]}\theta_{j-1,L}(x_{-\al_{i_j}})\eta_L. 
\] The statement then follows from the fact that $\eta_{j}^{2}=y_{j}\eta_j.$
\end{proof}
\begin{ex}
For $SL(4)$ whose simple roots are $\alpha_{1},\al_{2},\al_{3},$ let's consider Bott-Salmelson $\hat{X_{I}}=P_{\al_{1}}\times^{B} P_{\al_{2}}\times^{B} P_{\al_{3}}/B$. Then $h_{T}(\hat{X_{I}})$ is a polynomial algebra generated by $\eta_{1},\eta_{2},\eta_{3}$ with following quotient relation:
\begin{align*}
 &\eta_{1}^{2}=x_{-\al_{1}}\eta_{1},\\
&\eta_{2}^{2}=x_{-\al_{1}-\al_{2}}\eta_{1}+\frac{x_{-\al_{2}}-x_{\al_{1}-\al_{2}}}{x_{-\al_{1}}}\eta_{1}\eta_{2},\\
&\eta_{3}^{2}=x_{\al_{1}-\al_{2}-\al_{3}}\eta_{3}+\frac{x_{-\al_{3}-\al_{2}}-x_{2\al_{1}-\al_{2}-\al_{3}}}{x_{-\al_{1}}}\eta_{1}\eta_{3}+\frac{x_{\al_{3}}-x_{\al_{1}+\al_{2}-\al_{3}}}{x_{-\al_{1}-\al_{2}}}\eta_{2}\eta_{3}\\ &+(\frac{x_{-\al_{3}-x_{\al_{2}-\al_{3}}}}{x_{-\al_{2}}x_{-\al_{1}}}-\frac{x_{-\al_{3}}-x_{\al_{2}-\al_{1}-\al_{3}}}{x_{\al_{1}-\al_{2}}x_{-\al_{1}}})\eta_{1}\eta_{2}\eta_{3}.   \\
\end{align*}
\end{ex}
Let us consider some geometry information of $\hX$, and the  $T$-fixed points. We fix some notation first.
For any $L \subset[l]$, define
\[(\hX_{I})_{L} =\{[g_{1},g_{2},\ldots,g_{l}]\in \hX_I|~g_{j}\in B \text{ if } j\notin L, \text{ and }g_{i} \notin B \text{ if }j\in L\}\subset \hX_I,
\]
and 
\[v_j^L=\prod_{k\in L\cap [j]}s_{i_k}, \quad v^L:=v^L_l=\prod_{k\in L}s_{i_k}.\]
The following lemma will be used in the proof of Theorem \ref{thm:GKM}. 

\begin{lem}\label{lem:dist}If $I=(i_1,...,i_l)$ is a sequence such that $i_j$ are all distinct, then for any $L\subset[l]$ and $j\in L^c$, $v^L_{j-1}(x_{-\al_{i_j}})$ are all distinct.
\end{lem}
\begin{proof}
Suppose $j_{1},j_{2}\in L^{c} $ and $j_{1}<j_{2}$. Then $L\cap [j_1]\subseteq L\cap [j_2]$. There are two cases. 

Case 1: $L\cap [j_1]=L\cap [j_2]$.  Then 
\[v_{j_{1}-1}^{L}(x_{-\al_{i_{j_{1}}}})=\prod_{k\in L\cap [j_{1}]}s_{i_{k}}(x_{-\al_{i_{j_{1}}}}),
\]
and 
\[v_{j_{2}-1}^{L}(x_{-\al_{i_{j_{2}}}}) =\prod_{k\in L\cap [j_{2}]}s_{i_{k}}(x_{-\al_{i_{j_{2}}}})=\prod_{k\in L\cap [j_{1}]}s_{i_{k}}(x_{-\al_{i_{j_{2}}}}).
\] They are not equal since $\al_{i_{j_{1}}}\neq\al_{i_{j_{2}}}$. 

Case 2.  $L\cap [j_1]\subsetneq L\cap [j_2]$. Denote $M= (L\cap [j_2])\backslash (L\cap [j_1])$.  Then
\[v_{j_{1}-1}^{L}(x_{-\al_{i_{j_{1}}}})=\prod_{k\in L\cap [j_{1}]}s_{i_{k}}(x_{-\al_{i_{j_{1}}}}), \quad 
v_{j_{2}-1}^{L}(x_{-\al_{i_{j_{2}}}})=\prod_{k\in L\cap [j_{1}]}s_{i_{k}}(\prod_{k'\in M}s_{{i_{k'}}}(x_{-\al_{i_{j_{2}}}})).\]
By definition of the Weyl group action, 
\[\prod_{k'\in M}s_{{i_{k'}}}(x_{-\al_{i_{j_{2}}}})=x_{\prod_{k'\in M}s_{{i_{k'}}}({-\al_{i_{j_{2}}}})}=x_{-\al_{i_{j_{2}}}+\sum_{k'\in M}c_{k'}\al_{i_{k'}}},\quad c_{k'}\in \bbZ, 
\]which is different from $x_{-\al_{i_{j_{1}}}}$, since the set $\{-\al_{i_{j_1}}, -\al_{i_{j_2}}, \pm\al_{i_{k'}}|k'\in M\}$ is linearly independent. Thus $v_{j_{1}-1}^{L}(x_{-\al_{i_{j_{1}}}})$ and $v_{j_{2}-1}^{L}(x_{-\al_{i_{j_{2}}}})$ are not equal to each other. 
\end{proof}

The following lemma recalled from \cite[Proposition 2.6]{W04}, provides some geometric information of the Bott-Samelson variety, which is useful for our computation. 
\begin{lem}\label{lem:Wi}
\begin{enumerate}
\item 
The set $\hX_I^{T}$ of $T$-fixed points in $\hX_I$, consists of  $2^{l}$ points $$ [g_{1},g_{2},\ldots,g_{l}]$$ where $g_{j}\in \left\{e,s_{i_{j}}\right\}.$ Here we think of $s_{i_j}$ as in $W\cong N_G(T)/T$ and pick a preimage for $s_{i_j}$ in $N_G(T)\subset G$. Consequently, we have bijection of sets from the power set $\calP_l:=\mathcal{P}([l])$ to  $\hX_I^T$, 
\[
L\mapsto \pt_L:= [g_1,...,g_l], \quad g_j=\left\{\begin{array}{cc} s_{i_j}, &\text{ if }j\in L, \\
e, & \text{if } j\not \in L.\end{array}\right.
\]
\item The set $(\hX_{I})_{L}$ is a  $T$-orbit containing the fixed point $\pt_L$, and isomorphic  to the affine space of dimension $|L|$. The variety $\hX_I$ has a decomposition $\coprod_{L\in \mathcal{E}_I}(\hX_{I})_{L}$. 

\item Suppose $L,L'\subset [l]$.  then $\pt_L \in \overline{(\hX_I)_{L'}} $ if and only if $L \subset L'$.  The weights of the $T$-action  on the tangent space of $\overline{(\hX_I)_{L'}}$ at $\pt_L$ are 
\[
\{-v_j^L(\al_{i_j})|j\in L'\}. 
\]
\end{enumerate}
\end{lem}

\begin{ex} \label{ex:A2fixed}For the $A_2$-case, consider $\hX_{(1,2)}=P_1\times^B P_2/B$. There are four $T$-fixed points, denoted by $\{00, 01, 10, 11\}$, corresponding to $\{[e,e], [e,s_2], [s_1,e], [s_1,s_2]\}$, or $\emptyset, \{2\}, \{1\}, \{1,2\}$ as subsets of $[2]$. The weights of the tangent spaces of $\hX_{(1, 2)}$ at the four points are:
\[\begin{matrix}
00:&-\al_1 ,  -\al_2 & 01: &-\al_1 ,  \al_2\\
10:&\al_1 ,  -\al_1-\al_2 & 11: & \al_1 , \al_1+\al_2. 
\end{matrix}
\]
\end{ex}

We denote the set of functions on $\mathcal{E}_I=\hX_I^T$ with values in $S$ by $F(\mathcal{E}_I;S)$. It is  a free $S$-module with basis $f_L, L\in \mathcal{E}_I$ defined by $f_{L}(L')=\delta_{L,L'}$, and have a ring structure  given by \[
f_L\cdot f_{L'}=\de_{L,L'}f_L. \]
Moreover, we have \[h_T(\hX_I)\cong F(\mathcal{E}_I;S),\]
where the fixed point $\pt_L$ corresponds to the basis element $f_L$.

Denote $\bj^I:\hX_I^T\to \hX_I$. For each $L\subset[l] $, denote by $\bj_L^I$ the embedding of $\pt_L$ into $\hX_I$. Sometimes we will skip the superscript $I$ for simplicity. Then 
\[
\bj^*(f)=\sum_{L\subset[l]}\bj_L^*(f)f_L, \quad f\in h_T(\hX_I). 
\]
Denote 
\begin{align}\label{eq:xIL}
x_{I,L}=\prod_{1\le j\le l}v^L_j(x_{-\al_{i_j}}). 
\end{align}
We have
\begin{lem} \label{lem:injpoint}For any $L\subset[l]$, we have $\bj^*\bj_*(f_L)=x_{I, L}f_L$. 
\end{lem}
\begin{proof}	This follows from \cite[\S2.\textbf{A}8]{CZZ3}, and Lemma \ref{lem:Wi} concerning the weights of the tangent space of $\hX_I$ at the point $L$. 
\end{proof}

\begin{ex}Following Example \ref{ex:A2fixed}, with $I=(\al_1,\al_2)$, we have
\[
x_{I,00}=x_{-\al_1}x_{-\al_2}, \quad x_{I,10}=x_{\al_1}x_{-\al_1-\al_2}, \quad x_{I,01}=x_{-\al_1}x_{\al_2}, \quad x_{I,11}=x_{\al_1}x_{\al_1+\al_2}. 
\]
\end{ex}

\section{Restriction to T-fixed points}\label{sec:4}

In this section, we compute the restriction formula of the $\eta_L$ basis. 
We first compute the restriction formula of the image of the characteristic map. 

\begin{lem}Let $I$ be a sequence of length $l$, and $\bfc_I:S\to \hX_I$ be the characteristic map, then
 \[
\bj^*\bfc_I(u)=\sum_{L\subset[l]}v^L(u)f_L. 
\]
\label{lem:BSchar}
\end{lem}
\begin{proof}We prove it by induction on the length $l$ of $I$. If $I=(i_1)$, then there are two points in $(P_{i_1}/B)^T$, corresponding to $e$ and $s_1$ (or $\emptyset$ and $[1]$ as subsets of $[1]$). Then from \cite[\S10]{CZZ3} we have 
\[
\bj^* \bfc_I(u)=uf_e+s_{i_1}(u)f_{s_{i_1}}.
\]
So the conclusion follows.

 Now assume it holds for all sequence of length $\le l-1$, and assume $I=(i_1,...,i_l)$. Denote $I'=(i_1,...,i_{l-1})$ and $\sigma:\hX_{I'}\to \hX_I$ the zero section.  By induction assumption, for each $L'\subset[l-1]$, we have
\begin{align}\label{eq:l}
(\bj^{I'}_{L'})^*\bfc_{I'}(u)=v^{L'}_{l-1}(u). 
\end{align}

Concerning $L\subset [l]$, we  have two cases: 

Case 1:  $l\in L$. In this case, $\pt_L\not\in \sigma(\hX_{I'})$, so 
\begin{align}\label{eq:sig}
(\bj^I_L)^*\circ \sigma_*=0.
\end{align}
 Moreover, we have the following commutative diagram 
\[
\xymatrix{\pt\ar[r]^{\bj_L^I} \ar[rd]_{\bj^{I'}_{L\backslash \{l\}}}& \hX_I\ar[d]^p\\
                  & \hX_{I'},}
\]
that is, $p\circ \bj^I_L=\bj^{I'}_{L\backslash \{l\}}$, so 
\begin{align}\label{eq:bl}
(\bj^I_L)^*\circ p^*=(\bj^{I'}_{L\backslash \{l\}})^*.
\end{align}
Denote $\xi=\sigma_*(1)$, then by Lemma \ref{lem:char}, we have
\begin{align*}
(\bj_L^I)^*\circ \bfc_I(u)&=(\bj_L^I)^*[p^*\bfc_{I'}(s_{i_l}(u))+p^*\bfc_{I'}(\Delta_{-\al_{i_l}}(u))\cdot \xi]\\
&=(\bj_L^I)^*p^*\bfc_{I'}(s_{i_l}(u))+(
\bj_L^I)^*p^*\bfc_{I'}(\Delta_{-\al_{i_l}}(u))\cdot (\bj_L^I)^*(\sigma_*(1))\\
&\overset{\sharp_1}=(\bj^{I'}_{L\backslash \{l\}})^*\bfc_{I'}(s_{i_l}(u))\\
&\overset{\sharp_2}=v^{L\backslash \{l\}}_{l-1}\circ s_{i_l}(u)=v^L_l(u).
	\end{align*}
Here the identity $\sharp_1$ follows from \eqref{eq:sig} and \eqref{eq:bl}, and $\sharp_2$ follows from \eqref{eq:l}. 

Case 2:  $l\not \in L$. In this case, we can view $L\subset [l-1]$, so we have commutative diagrams:
\[
\xymatrix{\pt\ar[r]^{\bj^I_L}\ar[rd]_{\bj^{I'}_L} & \hX_I\ar[d]^p\\
            & \hX_{I'}}, \quad \xymatrix{\pt\ar[r]^{\bj^I_L}\ar[dr]_{\bj^{I'}_L} & \hX_I\\
                               & \hX_{I'}\ar[u]_{\sigma}},  
\] 
so  $p\circ \bj^I_L=\bj^{I'}_L$ and $\sigma\circ \bj^{I'}_L=\bj^I_L$. The latter implies that 
\begin{align}\label{eq:sig2}
(\bj^I_L)^*\sigma_*(1)=(\bj^{I'}_L)^*\sigma^*\sigma_*(1)\overset{\text{Lem.} \ref{lem:char}}=(\bj^{I'}_L)^*\bfc_{I'}(x_{-\al_{i_l}}). 
\end{align} Therefore, 
\begin{align*}
(\bj_L^I)^*(\bfc_I(u))&=(\bj_L^I)^*[p^*\bfc_{I'}(s_{i_l}(u))+p^*\bfc_{I'}(\Delta_{-\al_{i_l}}(u))\cdot \xi]\\
  &=(\bj_L^I)^*p^*\bfc_{I'}(s_{i_l}(u))+(\bj_L^I)^*p^*\bfc_{I'}(\Delta_{-\al_{i_l}}(u))\cdot (\bj_L^I)^*(\sigma_*(1))\\
  &=(\bj^{I'}_L)^*\bfc_{I'}(s_{i_l}(u))+(\bj^{I'}_L)^*\bfc_{I'}(\Delta_{-\al_{i_l}}(u))\cdot (\bj^{I'}_L)^*\bfc_{I'}(x_{-\al_{i_l}})\\
  &=(\bj^{I'}_L)^* \bfc_{I'}(s_{i_l}(u)+\frac{u-s_{i_l}(u)}{x_{-\al_{i_l}}}x_{-\al_{i_l}})\\
   &=(\bj^{I'}_L)^*\bfc_{I'}(u)\\
  &=v^{L}_{l-1}(u)=v^L_{l}(u). 
\end{align*}
The proof is finished. 
\end{proof}

Before computing the restriction formula of $\eta_L$, we first consider an example. 

\begin{ex} \label{ex:rank2} Consider the case of $A_2$. Let $\left\{\alpha_{1},\alpha_{2}\right\}$ be the set of simple roots. We consider the Bott-Samelson variety $\hX_I=P_{{1}}\times^{B} P_{{2}}/B$ for $I=(1,2)$. Following Example \ref{ex:A2fixed}, there are four torus-fixed points, denoted by $\calP_2=\{00, 01, 10, 11\}$.  Similarly, denote $(P_{1}/B)^T$ by $\calP_1=\{0,1\}$.  Denote $\bj^{I}:\mathcal{E} _I\hookrightarrow \hX_I$ and $\bj^1:\calP_1\hookrightarrow (P_{1}/B)^T$.  Consider the following commutative  diagram:
\[\xymatrix{
P_{{1}}\times^{B} P_{{2}}/B  \ar[d]^{p_{2}}  & \calP_2=\left\{00, 01, 10, 11\right\} \ar[l]^-{\bj^I} \ar[d]^{p_2'}\\ 
P_{{1}}/B \ar[d]^{p_{1}} 
\ar@/^/[u]^{\sigma_{2}} & \calP_1=\left\{0,1\right\} \ar[l]^{\bj^{1}}\\ 
 \pt \ar@/^/[u]^{\sigma_{1}}&
}.
\]
Here $\sigma_i$ are the zero sections, $p_2'$ is induced by the projection map $p_2$, so it maps $00, 01$ to $0$,  and $10$ and $11$  to $1$. Moreover, by definition, $\bj^1_0=\sigma_1$, and $ \sigma_2\circ\bj^1_{i}=\bj^I_{i0}$ for $i=0,1$.  We have 
\[
\eta_1=p_2^*\sigma_{1*}(1), \quad \eta_2=\sigma_{2*}(1), 
\]
and 
\[h_T((\hX_I)^T)=S\{f_{00}, f_{01}, f_{10}, f_{11}\}, \quad h_T((P_{1}/B)^T)=S\{f_0, f_1\}.\]
Denote $\bfc_1:S\to h_T(P_1/B)$. 

First of all, from the definition of $p_2'$ and \eqref{eq:finite}, we know 
\[
p_2'^*(f_0)=f_{00}+f_{01}, \quad p_2'^*(f_1)=f_{10}+f_{11}. 
\]
Moreover, since $\bj_0^1$ coincides with $\sigma_1$ and $\bj^1_1(\pt)\not \in \sigma_1(\pt)$, so $(\bj^1_1)^*\sigma_{1*}=0$ and  \[(\bj^1)^*\sigma_{1*}(1)=(\bj^1_0)^*\sigma_{1*}(1)=\sigma_1^*\sigma_{1*}(1)=x_{-\al_1}f_0,\]
where the last identity follows from the fact that the tangent space of $P_1/B$ at $0$ has weight $-\al_1$. Hence, 
\begin{align}\label{eq:eta1}
(\bj^I)^*(\eta_1)=(\bj^I)^*p_2^*\sigma_{1*}(1)=p_2'^{*}(\bj^1)^*\sigma_{1*}(1)=p_2'^{*}(x_{-\al_1}f_0)=x_{-\al_1}(f_{00}+f_{01}). 
\end{align}

We then compute $(\bj^I)^*(\eta_2)$, by using the identity
\[
(\bj^I)^*(\eta_2)=\sum_{x\in \calP_3}(\bj_x^I)^*(\eta_2)f_x. 
\]
 Since $01, 11\not\in \sigma_2(P_{1}/B)$, so $(\bj^I_{01})^*(\eta_2)=(\bj^I_{11})^*(\eta_2)=0$. From Lemma \ref{lem:char}, we know that $\sigma_2^*\sigma_{2*}(1)=\bfc_1(x_{-\al_2})$. So 
\[
(\bj^I_{00})^*(\eta_2)=(\bj^I_{00})^*\sigma_{2*}(1)=(\bj^1_0)^*\sigma_2^*\sigma_{2*}(1)=(\bj^1_0)^*(\bfc_1(x_{-\al_2}))\overset{\sharp}=x_{-\al_2},
\]
where $\sharp$ follows from Lemma \ref{lem:BSchar}. 
 Similarly, from $\bj^I_{10}=\sigma_2\circ \bj^1_1$, we have
\[
(\bj^I_{10})^*(\eta_2)=(\bj^I_{10})^*\sigma_{2*}(1)=(\bj^1_1)^*\sigma_2^*\sigma_{2*}(1)=(\bj^1_1)^*(\bfc_1(x_{-\al_2}))=s_1(x_{-\al_2})=x_{-\al_1-\al_2}.
\]
 Therefore, 
\begin{align}\label{eq:eta2}
(\bj^I)^*(\eta_2)=x_{-\al_2}f_{00}+x_{-\al_1-\al_2}f_{10}. 
\end{align}
\end{ex}

Now we compute the restriction formula of $\eta_L$. 
\begin{thm}\label{thm:res}
Let $I$ be a sequence of length $l$. For any two subsets $L,M\subset [l]$, denote $L^c=[l]\backslash L$, and 
\[
a_{L,M}=\prod_{k\in L}v_{k-1}^{M}(x_{-\al_{i_{k}}}). 
\]
Then 
\begin{align*}
\bj^*(\eta_L)=\sum_{M\subset L^c}a_{L,M}f_{M}. 
\end{align*}
\end{thm}

\begin{proof}We first consider $L=\{k\}$, and prove the following identity
\[
\bj^*(\eta_k)=\sum_{M\subset  L^c}v^M_{k-1}(x_{-\al_{i_k}})f_M.
\]
 Denote $I_k=(i_1,...,i_k)$, and similarly denote $I_{k-1}$. Firstly, we compute $(\bj^{I_k}_M)^*\sigma_{k*}(1)$ for each $M\subset[k]$, with $\sigma_k:\hX_{I_{k-1}}\to \hX_{I_k}$. 

If $k\in M$, then the point $\bj^{I_k}_M(\pt)\not\in \sigma_k(\hX_{I_{k-1}})$, so 
\[(\bj^{I_k}_M)^*\sigma_{k*}(1)=0.\]
 If $k\not\in M$, then $M\subset[k-1]$, $v^M_k=v^M_{k-1}$, and  we have the following commutative diagram
\[
 \xymatrix{           \pt\ar[r]^{\bj^{I_k}_M}\ar[dr]_{\bj^{I_{k-1}}_{M}}                           &\hX_{I_k}\\
                                       &\hX_{I_{k-1}}  \ar[u]_{\sigma_k}.  }
\]
Therefore, 
\begin{align}\label{eq:k}
(\bj^{I_k}_M)^*\sigma_{k*}(1)=(\bj^{I_{k-1}}_M)^*\sigma_k^*\sigma_{k*}(1)\overset{\text{Lem.} \ref{lem:char}}=(\bj^{I_{k-1}}_M)^*\bfc_{I_{k-1}}(x_{-\al_{i_k}})\overset{\text{Lem.} \ref{lem:BSchar}}= v^{M}_{k-1}(x_{-\al_{i_k}}). 
\end{align}
Here $\bfc_{I_{k-1}}$ is the characteristic map on $\hX_{I_{k-1}}$. 

Now consider the following commutative diagram
\[
\xymatrix{(\hX_I)^T\ar[r]^{\bj^I}\ar[d]^{p'}& \hX_I\ar[d]^p\\
(\hX_{I_k})^T\ar[r]^{\bj^{I_k}} &\hX_{I_k}.}
\]
We have 
\begin{align*}
(\bj^I)^*(\eta_k)&=(\bj^I)^*p^*(\sigma_{k*}(1))=p'^*(\bj^{I_k})^*\sigma_{k*}(1)\\
&=p'^*[\sum_{M\subset[k]}(\bj^{I_k}_M)^*\sigma_{k*}(1)f_M]\\
&\overset{\eqref{eq:k}}=p'^*[\sum_{M\subset [k-1]}v^M_{k-1}(x_{-\al_{i_k}})f_M]\\
&\overset{\eqref{eq:finite}}=\sum_{M\subset [k-1]}v^M_{k-1}(x_{-\al_{i_k}})\sum_{M'\subset \{k+1,...l\}}f_{M\cup M'}\\
&=\sum_{M\subset( [l]\backslash \{k\})}v^M_{k-1}(x_{-\al_{i_k}})f_M. 
\end{align*}
So the case $L=\{k\}$ is proved. 

Now for a general subset $L\subset[l]$, we have
\[
\bj^*(\eta_L)=\prod_{k\in L}\bj^*(\eta_k)=\prod_{k\in L}\sum_{M\subset ([l]\backslash \{k\})}v^M_{k-1}(x_{-\al_{i_k}})f_M=\sum_{M\subset L^c}\prod_{k\in L}v^M_{k-1}(x_{-\al_{i_k}})f_M.
\]
\end{proof}

For $I$ of length $l$ and $L\subset[l]$, note  the difference between 
\[a_{[l], L}=\prod_{1\le k\le l}v^L_{k-1}(x_{-\al_{i_k}}), \quad x_{I,L}=\prod_{1\le k\le l}v^L_{k}(x_{-\al_{i_k}}). 
\]
They are only related when $L=[l]$, in which case we have
\[
a_{[l], [l]}=\prod_{1\le k\le l}v^{[l]}_{k-1}(x_{-\al_{i_k}})=\prod_{1\le k\le l}v^{[l]}_{k}(x_{\al_{i_k}}),\quad 
x_{I,[l]}=\prod_{1\le k\le l}v^{[l]}_{k}(x_{-\al_{i_k}}). 
\]

\begin{cor}\label{cor:inj}The map $\bj^{*}:h_T(\hX_I)\to h_T(\hX_I^T)$ is an injection. 
\end{cor}
\begin{proof}It follows from Theorem \ref{thm:res} that 
\[
\bj^*(\eta_L)=\sum_{M\subset L^c}a_{L,M}f_M.
\]
So if we order $\{\bj^*(\eta_L)|L\subset[l]\}, \{f_M|M\subset [l]\}$ by inclusion of subsets $L'\subset L$, then the transition matrix from $f_M$ to $\bj^*(\eta_L)$ will be skew-triangular. Moreover, the entries on the skew-diagonal will be 
\[
a_{L,L^c}=\prod_{k\in L}v^{L^c}_{k-1}(x_{-\al_{i_k}}),
\]
which is regular in $S$. Therefore, $\bj^*$ is injective. 
\end{proof}

\begin{thm} \label{thm:image}Let $I$ be a sequence of length $l$. Then
\[
\im \bj^*\subset \{\sum_{L\subset [l]}a_Lf_L|\frac{a_{L_1}-a_{L_2}}{v^{L_1}_{k-1}(x_{-\al_{i_k}})}\in S, ~\forall L_1,L_2 \text{ such that }L_1=L_2\sqcup \{k\}\}. 
\]
Here $\sqcup$  denotes the disjoint union. 
\end{thm}
\begin{proof} Denote the right hand side by $\Psi$.  We first show that $\Psi$ is a ring. It is clearly additively closed. For the multiplication, consider  
\[f=\sum_{L\subset[l]}a_Lf_L, \quad g=\sum_{L'\subset [l]}b_{L'}f_{L'} \in \Psi, 
\] then 
\[fg=\sum_{L, L'\subset [l]}\de_{L,L'}a_{L}b_{L'}f_L=\sum_{L\subset[l]}a_Lb_Lf_L.\]
 For any $L_1,L_2$ such that $L_1=L_2\sqcup \{k\}$, by definition we have $v^{L_1}_{k-1}=v^{L_2}_{k-1}$, so $v^{L_1}_{k-1}(x_{-\al_{i_k}})=v^{L_2}_{k-1}(x_{-\al_{i_k}})$. Therefore, 
\[
a_{L_1}b_{L_1}-a_{L_2}b_{L_2}=(a_{L_1}-a_{L_2})b_{L_1}-(b_{L_2}-b_{L_1})a_{L_2},
\]
so is divisible by $v^{L_1}_{k-1}(x_{-\al_{i_k}})$. Therefore, $f g\in \Psi$.

We then show that $\im \bj^*\subset \Psi$. Since $\bj^*$ is multiplicative, it suffices to show 
\[\bj^*(\eta_m)=\sum_{L\subset ([l]\backslash \{m\})}v^L_{m-1}(x_{-\al_{i_m}})f_L\]  belongs to the RHS. Suppose $L_1=L_2\sqcup\{k\}$. Clearly $k\neq m$.  If $k> m$, then by definition,  $v^{L_1}_{m-1}=v^{L_2}_{m-1}$, so $v^{L_1}_{m-1}(x_{-\al_{i_m}})=v^{L_2}_{m-1}(x_{-\al_{i_m}})$, which implies that $\bj^*(\eta_m)\in \Psi.$ 

If $k<m $, denote
\[L_1\cap [m-1]=\{j_1<j_2<\cdots j_{t}<k<j_{t+1}<\cdots<j_s\},\]
\[L_2\cap [m-1]=\{j_1<j_2<\cdots j_{t}<\hat{k}<j_{t+1}<\cdots<j_s\},\]
 (in other words, $k$ is omitted in $L_2$). Then 
\begin{align*}
&v^{L_1}_{m-1}(x_{-\al_{i_m}})-v^{L_2}_{m-1}(x_{-\al_{i_m}})\\
&=s_{i_{j_1}}s_{i_{j_2}}\cdots s_{i_{j_t}}s_{i_k}s_{i_{j_{t+1}}}\cdots s_{i_{j_s}}(x_{-\al_{i_m}})-s_{i_{j_1}}s_{i_{j_2}}\cdots s_{i_{j_t}}\widehat{s_{i_k}}s_{i_{j_{t+1}}}\cdots s_{i_{j_s}}(x_{-\al_{i_m}})\\
&=v^{L_1}_{k-1}s_{i_k}(s_{i_{j_{t+1}}}\cdots s_{i_{j_s}})(x_{-\al_{i_m}})-v^{L_1}_{k-1}(s_{i_{j_{t+1}}}\cdots s_{i_{j_s}})(x_{-\al_{i_m}}).
\end{align*}
According to Lemma \ref{lem:div}, this is  divisible by $v^{L_1}_{k-1}(x_{-\al_{i_k}})$. 
\end{proof}

We can strengthen the conclusion in some cases. The proof essentially uses the fact that the transition matrix from $f_L, L\subset[l]$ to $\bj^*(\eta_M), M\subset[l]$ is skew-triangular, following from Theorem \ref{thm:res}. 

\begin{thm}\label{thm:GKM}If $I=(i_1,...,i_l)$ with $i_j$ all distinct, then 
we have equality in Theorem \ref{thm:image}. 
\end{thm}

\begin{proof}
It suffices to show that  $\Psi\subset \im \bj^*$. Suppose
\[f=\sum_{L\subset[l]}a_L f_L\in\Psi, \quad \text{with }a_L=0 \text{ unless }L=\emptyset,\]
  then for any $k\in [l]$,  $a_\emptyset=a_\emptyset-a_{\{k\}}$ is divisible by $v^\emptyset_{k-1}(x_{-\al_{i_k}})=x_{-\al_{i_k}}$. Since $x_{\al_{i_j}}, 1\le j\le l$ are all distinct, by \cite[Lemma 2.7]{CZZ2}, we see that  $a_\emptyset$ is divisible by $\prod_{k\in [l]}x_{-\al_{i_k}}$. Note that by Theorem \ref{thm:res}, 
\[\bj^*(\eta_{[l]})=\prod_{k\in [l]}x_{-\al_{i_k}}f_\emptyset,\]
so $f$ is a multiple of $\bj^*(\eta_{[l]})$, i.e., $f\in \im \bj^*$. 

Assume the conclusion holds for any $f$ that can be written as a linear combination of $f_L$ with $|L|\le t-1$. Now let 
\[f=\sum_{L\subset [l]}a_Lf_L\in \Psi, \quad \text{with }a_L=0 \text{ unless }|L|\le t.\] Let $L_0$ be a subset of $[l]$ of cardinality $t$. For any $k\in L_0^c$, we have $a_{L_0\sqcup\{k\}}=0$, so $v_{k-1}^{L_0}(x_{-\al_{i_k}})|a_{L_0}$ Now from Theorem \ref{thm:res} we know 
\[
\bj^*(\eta_{L_0^c})=\sum_{M\subset L_0}a_{L_0^c,M}f_M,\quad a_{L_0^c, L_0}=\prod_{j\in L_0^c}v_{j-1}^{L_0}(x_{-\al_{i_j}}). \]
By Lemma \ref{lem:dist}, we know that   $v^{L_0}_{j-1}(x_{-\al_{i_j}})$ are all distinct for $j\in L_0^c$.  By \cite[Lemma 2.7]{CZZ2}, we know that $a_{L_0^c, L_0}|{a_{L_0}}$. Write $a_{L_0}=c_{L_0}a_{L_0^c, L_0}$ with $c_{L_0}\in S$. Therefore, 
\[
f':=f-\sum_{|L_0|=t}c_{L_0}\bj^*(\eta_{L_0^c})=\sum_{|L|< t}a'_Lf_L,
\]
By induction hypothesis, $f'\in \im \bj^*$. Therefore, $f\in \im \bj^*$. The proof is finished. 
\end{proof}

\begin{rem}Let $T_{i}$ be the subtorus of rank 1 corresponding to $\al_{i}$, 
i.e., $T_i=(\ker \al_i)^\circ$ where $\al_i$ is viewed as a character $T\to 
k^*$. If $I=(i_1,...,i_l)$ is a sequence such that $i_j$ are all distinct, it is 
not difficult to see that for 
any $1\le k\le l$, 
\[\hX_I^{T_{i_k}}=\{[g_1,...,g_l]|g_jB\in \{B,s_{i_j}B\} ~
\forall ~j\neq 
k\},\]
and 
\[
\hX_I^{T'}=\{[g_1,...,g_l]|~g_jB\in \{B,s_{i_j}B\} ~\forall j\}
\]
if $T'$ is any subtorus of corank 1 different from $T_{i_j}, j=1,...,l$. In 
other words, for any subtorus of corank 1, the irreducible components of the 
invariant subvariety has dimension at most one. This corresponds to the 
so-called Goresky-Kottwitz-MacPherson (GKM) condition. In other words, in this 
case, 
the Bott-Samelson variety is a GKM space. This corresponds to the conclusion of 
Theorem \ref{thm:GKM}. 

On the other hand, if $P_{i_j}$ are not distinct, the space $\hX_I$ will not be 
GKM. For instance, if $I=(1,2,1)$, the $T_1$-fixed subspace contains the 
following subset
\[
\{[g_1,e,g_1']|g_i, g_i'\in P_1\},
\]
so the dimension condition is not satisfied. For more detailed discussion of GKM 
spaces, see \cite{GKM98, GHZ06}. 
\end{rem}

\section{Push-forward to cohomology of flag varieties}\label{sec:5}

In this section, we compute the push-forward of the basis $\eta_L$ along the canonical map $q_I:\hX_I\to G/B$, which generalizes the computation of Bott-Samelson classes in \cite{CZZ3}. 

 Recall that the $T$-fixed points of $G/B$ is in bijection to $W$, so we have 
\[h_T((G/B)^T)\cong \oplus_{w\in W}S. \] 
Denote by $f_w\in h_T(W)$ the basis element corresponding to $w\in W$. Denote $i:W\to G/B$ to be the embedding, and denote $\pt_e=(i|_e)_*(1)\in h_T(G/B)$. Let $\pi_i:G/B\to G/P_i$ be the canonical map, and 
denote $A_i=\pi_{i}^*\circ \pi_{i*}: h_T(G/B)\to h_T(G/B)$. For any sequence $I$, denote by $I^{\rev}$ the sequence obtained by reversing $I$. 

\begin{prop} \cite[Lemma 7.6]{CZZ3} \label{prop:BS}
For any sequence $I$,  we have  ${q_{I}}_{*}(1)=A_{I^{\rev}}(\pt_e)$.
\end{prop}

 The following is an easy generalization of Proposition \ref{prop:BS}. 
\begin{thm}\label{thm:eta_{k}}
Let $I$ be a sequence of length $l$ and $1\le k\le l$. Denote by $I_k$ the subsequence of $I$ obtained by removing the $k$-th term from $I$.  Then $(q_I)_*(\eta_k)=A_{I_k^{\rev}}(1)$. 
\end{thm}

\begin{proof}

Denote the  sequence by $I=(i_1,...,i_l)$. For any $k\le l$, denote
\[
\xymatrix{
P_{{i_{1}}}\times^{B} P_{{i_{2}}}\times^{B}\times...\times^{B} P_{{i_{k}}}/B \ar[r]^<<<<{q_{k}} \ar[d]^<<<<{p_k} & G/B  \\P_{{i_{1}}}\times^{B} P_{{i_{2}}}\times^{B}\times...\times^{B} P_{{i_{k-1}}}/B  \ar[r]^<<<<{q_{k-1}} \ar@/^/[u]^{\sigma_{k}}   & G/B.
}
\]
Note that $q_I=q_l$ and $q_k\circ \sigma_k=q_{k-1}$. Denote by $p$ the composition $p_{k+1}\cdots p_l$.
By using the base change formula from diagram \eqref{eq:basechange}, we have 
\begin{align*}     
{q_{I}}_{*}(\eta_{k})    &={q_{l}}_{*}p^{*}({\sigma_{k}}_{*}(1)) \\ &=q_{l*}p_l^*p_{l-1}^*\cdots p_{k+1}^*\sigma_{k*}(1)\\
&=\pi_{\al_{i_l}}^*\pi_{\al_{i_l}*}(q_{l-1})_*p_{l-1}^*\cdots p^*_{k+1}\sigma_{k*}(1)\\
&=(\pi_{\al_{i_l}}^*\pi_{\al_{i_l}*})(\pi_{\al_{i_{l-1}}}^*\pi_{\al_{{i_{l-1}*}}})\cdots (\pi_{\al_{i_{k+1}}}^*\pi_{\al_{i_{k+1}}*})q_{k*}\sigma_{k*}(1)\\
&=A_{i_l}A_{i_{l-1}}\cdots A_{i_{k+1}}q_{{k-1}*}(1)\\
&=A_{i_l}A_{i_{l-1}}\cdots A_{i_{k+1}}A_{i_{k-1}}\cdots A_{i_1}(\pt_e)=A_{I_k^{\rev}}(\pt_e). 
\end{align*}
\end{proof}

To compute $q_{I*}(\eta_L)$ for general $L\subset[l]$, we need the following lemma.
\begin{lem} \label{lem:pushf}For any $L\subset[l]$, we have
\[
\eta_L=\sum_{L_1\subset [l]}\frac{a_{L,L_1}}{x_{I,L_1}}\bj_*(f_{L_1}),
\]
where $a_{L,L_1}$ are defined in Theorem \ref{thm:res}. Note that the coefficients in this formula belong to   $Q:=S[\frac{1}{x_\al}|\al\in \Sigma]$.
\end{lem}
\begin{proof}By Corollary \ref{cor:inj}, we know that $\bj^*(\eta_L)$ becomes a basis of $Q\otimes_Sh_T(W)$. In other words, $\bj^*$ induces an isomorphism
\[
\bj^*:Q\otimes_S h_T(\hX_I)\to Q\otimes_S h_T(W).
\]
Moreover, by Lemma \ref{lem:injpoint}, we know that $\bj_*$ is the inverse of the $\bj^*$ (after tensoring with $Q$). Therefore, $\bj_*(f_L)$ is a $Q$-basis of $Q\otimes_Sh_T(\hX_I)$. 
Denote
\[
\eta_L=\sum_{L_1\subset [l]}b_{L,L_1}\bj_*(f_{L_1}), \quad b_{L, L_1}\in Q.
\]
Then by Theorem \ref{thm:res} and Lemma \ref{lem:injpoint}, we have
\begin{align*}
\sum_{L_2\subset L^c}a_{L,L_2}f_{L_2}=\bj^*(\eta_L)=\sum_{L_1\subset[l]}b_{L,L_1}\bj^*\bj_*(f_{L_1})=\sum_{L_1\subset[l]}b_{L,L_1}x_{I,L_1}f_{L_1}.
\end{align*}
Therefore, $b_{L,L_1}=\frac{a_{L,L_1}}{x_{I,L_1}}$.
\end{proof}
 The following is the main result of this section, which computes the push-forward of $\eta_L$ to cohomology of $G/B$. 

\begin{thm}\label{thm:pushf}For any  sequence $I=(i_1,...,i_l)$, we have
\[
i^*q_{I*}(\eta_L)=\sum_{L_1\subset L^c}\frac{a_{L,L_1}\cdot v^{L_1}(x_\Pi)}{x_{I,L_1}}f_{v^{L_1}}, \quad x_\Pi:=\prod_{\al<0}x_\al\in S.
\]
Note that a priori the coefficients of $f_{v^{L_1}}$ belong to $S$. 
\end{thm}
\begin{proof}
Consider the following commutative diagram
\[
\xymatrix{\hX_I^T\ar[r]^{\bj}\ar[d]^{q'} & \hX_I\ar[d]^{q_I}\\
           W \ar[r]^i & G/B. }
\]
Note that by definition, $q'$ maps the point corresponding to $L\subset[l]$ to $v^L\in W$. Therefore, 
\[
q'_*(f_L)=f_{v^L}\in h_T(W).
\]
Firstly, we have 
\begin{align*}
i^*q_{I*}\bj_*(f_L)&=i^*i_*q'_*(f_L)=i^*i_*(f_{v^L})=v^{L}(x_{\Pi}) f_{v^L}, 
\end{align*}
where the last identity follows from \cite[Corollary 6.4]{CZZ3}. 
 Consequently, by Lemma \ref{lem:pushf}, we have
\begin{align*}
i^*q_{I*}(\eta_L)=i^*q_{I*}\sum_{L_1\subset L^c}\frac{a_{L,L_1}}{x_{I,L_1}}\bj^*(f_{L_1})=\sum_{L_1\subset L^c}\frac{a_{L,L_1}\cdot v^{L_1}(x_{\Pi})}{x_{I,L_1}}f_{v^{L_1}}. 
\end{align*}
\end{proof}

\begin{rem}In case $\eta_L=\eta_\emptyset$ or $\eta_k$, as in Proposition \ref{prop:BS} and Theorem \ref{thm:eta_{k}}, one can express $q_{I*}(\eta_L)$ as the operators $A_i$ applied on $\pt_e$. By using the method of formal affine Demazure algebra, started in \cite{KK86, KK90} and continued in \cite{CZZ1, CZZ2, CZZ3}, one will obtain a restriction formula of $i^*q_{I*}(\eta_L)$.  Roughly speaking, there is an algebra $\bfD_F$ generated by algebraic analogue of the push-pull operators $A_i$, whose dual is isomorphic to $h_T(G/B)$. The algebra $\bfD_F$ acts on $h_T(G/B)$, via two actions (denoted by $\bullet$ and $\odot$ in \cite{LZZ16}). Both actions will give restriction formulas of $A_I(\pt_e)$. Indeed, by using the two actions, one will obtain two different, but equivalent formulas, one of which coincides with the one given by Theorem \ref{thm:pushf}. 
\end{rem}

\begin{cor}\label{cor:pushBott}
Let $I$ be any sequence of length $l$. For any $L\subset[l]$, denote by $q_L:\hX_L\to G/B$.  Then $q_{I*}(\eta_{L})=q_{L^{c}*}(1)$. 
\end{cor}
\begin{proof}
From Theorem \ref{thm:pushf} we have \begin{align}\label{eq:one} i^*q_{I*}(\eta_L)&=\sum_{L_1\subset L^c}\frac{a_{L,L_1}\prod_{\al<0}v^{L_1}(x_{\al})}{x_{I,L_1}}f_{v^{L_1}}, 
\end{align}
\begin{align}\label{eq:two}
    i^{*}{q_{L^{c}}}_{*}(1)=\sum_{L_1\subset L^c}\frac{\prod_{\al<0}v^{L_1}(x_\al)}{x_{L^c, L_1}}f_{v^{L_1}}.
\end{align}
By definition 
\[x_{I,L_{1}}=\prod_{j\in I}v_{j}^{L_{1}}(x_{-\al_{i_{j}}}), \quad x_{L^{c},L_{1}}=\prod_{j\in L^{c}}v_{j}^{L_{1}}(x_{-\al_{i_{j}}}), \quad a_{L,L_1}=\prod_{j\in L}v^{L_1}_{j-1}(x_{-\al_{i_j}}). 
\]
Since $L\cap L_{1}=\emptyset$, so for any $j\in L$, $v^{L_1}_j=v^{L_1}_{j-1}$, and we have
 \begin{align*}
    x_{I,L_{1}}&=\prod_{j\in L^{c}}v_{j}^{L_{1}}(x_{-\al_{i_{j}}})\prod_{j\in L}v_{j}^{L_{1}}(x_{-\al_{i_{j}}})\\ &=\prod_{j\in L^{c}}v_{j}^{L_{1}}(x_{-\al_{i_{j}}})\prod_{j\in L}v_{j-1}^{L_{1}}(x_{-\al_{i_{j}}})\\ &=x_{L^c, L_1}a_{L,L_{1}} .\end{align*} 
Therefore, $i^*q_{I*}(\eta_L)=i^*q_{L^c*}(1)$. By \cite[Theorem 8.2]{CZZ3}, we know $i^*$ is injective. So $q_{I*}(\eta_L)=i^*q_{L^c*}(1)$. 
\end{proof}

By using this result, we can derive the Chevalley formula for equivariant oriented cohomology. For each $w\in W$, we fix a reduced sequence $I_w$, then the Bott-Samelson class $\zeta_{I_w}$ is defined to be the push-forward class along the map $q_{I_w}:\hX_{I_w}\to G/B$, i.e., $\zeta_{I_w}=q_{I_w*}(1)$. It is proved in \cite[Proposition 8.1]{CZZ3} that $\{\zeta_{I_w}|w\in W\}$ is a basis of $h_T(G/B)$. Denote the characteristic maps from $h_T(\pt)$ to $G/B$ and to $\hX_{I_w}$ by $\bfc'$ and $\bfc_{I_w}$, respectively. By definition, $\bfc_{I_w}=q_{I_w}^*\bfc'$.
\begin{cor}[Chevalley Formula]
 For any $u\in h_{T}(pt)$, we have \[\bfc'(u)\cdot \zeta_{w}=\sum_{L\subset[\ell(w)]}\theta_{I,L}(u)\zeta_{L^c},\]
where $\zeta_{L^c}=q_{L^c*}(1)$ and $\theta_{I,L}(u)$ was defined in Lemma \ref{lem:ind}. 
\begin{proof}
We have \[q_{I_w*}(\bfc_{I_w}(u))=q_{I_w*}(\bfc_{I_w}(u)\cdot 1)=q_{I_w*}(q_{I_w}^{*}(\bfc'(u))\cdot 1)=\bfc'(u)\zeta_{I_w},\]
where the last identity follows from the projection formula. 
Then Lemma \ref{lem:ind} and Corollary \ref{cor:pushBott} imply that 
\begin{align*}q_{I_w*}(\bfc_{I_w}(u)) & =\sum_{L\subset [\ell(w)]}\theta_{I_w,L}(u)q_{I_w*}(\eta_{L})\\
 & = \sum_{L\subset[\ell(w)]}\theta_{I_w,L}(u)q_{L^c*}(1)\\
&= \sum_{L\subset[\ell(w)]}\theta_{I_w,L}(u)\zeta_{L^c}.
\end{align*}
The conclusion then follows.
\end{proof}
\end{cor}
\newcommand{\arxiv}[1]
{\texttt{\href{http://arxiv.org/abs/#1}{arXiv:#1}}}

\bibliographystyle{plain}

\end{document}